\documentclass[journal]{IEEEtran}
\ifCLASSINFOpdf
\else
\fi

\usepackage{graphicx}          

\usepackage[dvips]{epsfig}    
\usepackage{latexsym,amssymb}
\usepackage{amsmath, amsbsy, amsthm}
\usepackage{amsopn, amstext}
\usepackage{colortbl}
\usepackage{cite}
\newtheorem{remark}{Remark}
\newtheorem{theorem}{Theorem}
\newtheorem{definition}{Definition}
\newtheorem{lemma}{Lemma}

\newtheorem{problem}{Problem}
\newtheorem{assumption}{Assumption}
\newtheorem{corollary}{Corollary}
\allowdisplaybreaks
\begin{document}
%
\title{Distributed Control and Stabilization for Discrete-time Large Scale Systems With Imposed Constraints}
%
%
%

\author{Qingyuan~Qi,~\IEEEmembership{}
        Huanshui~Zhang,~\IEEEmembership{Senior Member,~IEEE}, and
        Peijun~Ju~\IEEEmembership{}
\thanks{This work is supported by the National Natural Science Foundation of
China under Grants 61120106011, 61573221, 61633014.

Q. Qi and H. Zhang are with School of Control Science and Engineering, Shandong University, Jinan 250061, 
China. P. Ju is with School of Mathematics and Statistics, Taishan University, Taian 271021, China. H. Zhang is the corresponding author. (e-mail: qiqy123@163.com, hszhang@sdu.edu.cn, jpj615@163.com.)}}

\maketitle

\begin{abstract}
This paper is concerned with the distributed control and stabilization problems for linear discrete-time large scale systems with imposed constraints. The main contributions of this paper are: Firstly, by using the maximum principle (necessary condition) for the finite horizon optimal control developed in this paper, the optimal distributed controller is thus derived, which can be easily calculated; Secondly, by defining the Lyapunov function with the optimal cost function, we show that the systems with imposed constraints can be stabilized by the optimal distributed controller for the infinite horizon case. The main techniques adopted in this paper are the maximum principle and the solution to the forward and backward difference equations (FBDE), which are the basis in solving distributed control and stabilization problems for linear systems with constraints.
\end{abstract}


%
\IEEEpeerreviewmaketitle

\section{Introduction}

With the rapid development of science and technology, large scale systems can be used to describe many kinds of physical systems, such as power networks, economic systems, urban traffic systems, and sensor networks \cite{dd78,sv78,bpm}. The study of the large scale systems is challenging on the controller design and system analysis. For large scale systems, the traditional centralized control strategy is not applicable, which is due to the lack of centralized information, on the other hand, the computational burden grows higher with the increase of the system dimensionality \cite{pp13,an12,hw68,fcwls}. Instead, distributed control strategy is an attractive technology in handing the control problems for large scale systems, \cite{mm1,mm2,cek,jm85}. Distributed control strategy can decompose large-scale systems to local sub-systems for which the local controller are designed to regulate the systems.

The study on distributed control problems have attracted many researchers' interest since the last century. There are many techniques to design a distributed control, such as modified LQR control, distributed optimal control and market-based control, see \cite{sv78,tc11,llm, rad,lsacc}. While for large-scale systems, the information exchange between the subsystems and the the controller is usually described as various constraints. Particularly, sparsity constraints and delay constraints between subsystems are most considered, \cite{dd78,an12,st13,ad10,rmls}. The complexity of such problems have been illustrated in previous literatures. In some special cases, such as the systems has a compatible sparsity pattern \cite{an12}, or the constraints are imposed on the closed-loop behavior \cite{tc11}. However, for the cases mentioned above, it has been shown that most of these types of distributed controllers are qualitatively different from the corresponding centralized controllers, and have been proved to be NP-hard problems \cite{hw68,jm85}.


Recently, a control structure is put forward, motivated by the coordinated control of
networks of wind turbines \cite{mm1,mm2}. The authors studied a homogeneous group of autonomous
agents with a global linear constraint on their average behavior. The constraints are imposed
on an virtual `average' agent, and such kind of formulation is motivated by certain control tasks arising in wind
power plants, which has great practical significance \cite{mm1}. Both \cite{mm1} and \cite{mm2} investigated the continuous-time case. It should be pointed out that the formulation investigated in this paper can be viewed as the discrete version of low-rank coordination problem. However, this paper differs from \cite{mm1,mm2} in the following aspects: 1) The stabilization and optimal distributed control of the infinite horizon will be investigated in this paper, and the necessary and sufficient stabilization conditions will be explored; 2) In this paper, we will study the discrete-time case; 3) The adopted methods in this paper are maximum principle and the solution to the FBDE, which are different from \cite{mm1,mm2} and are more intuitive; 4) Our methods have high scalability, and can be easily extended to systems with time delay and random disturbances.

As far as we know, the distributed control and stabilization problems for large scale systems with constraints have not been fully solved yet, especially the optimal distributed controller design and the stability conditions haven't been developed \cite{mm1,mm2}. However, we overcome the obstacle for the considered problem in this paper. This paper provides a thorough solution to the discrete-time distributed control and stabilization problems with low-rank coordination. First, the maximum principle (necessary condition) of the optimal controller will be developed for the first time, and the system of the FBDE will be obtained. The maximum principle will serve as the basis in solving the distributed control and stabilization problems under consideration. Next by decoupling the FBDE, the optimal distributed controller for the finite horizon case will be derived and the relationship between the costate and the state (solution to the FBDE) will be obtained. Finally, under mild conditions, we will investigate the control and stabilization problems for the infinite horizon case, and the necessary and sufficient stabilization conditions will be developed. We will show the stabilizing controller also minimizes the associated infinite horizon cost function.

Our presentation is organized as follows. In section II, we give the problem formulation and the finite horizon distributed control problem will be solved. In section III, we present and discuss the optimal control of infinite horizon case, and the necessary and sufficient stabilizing conditions will be explored. A numerical example is given in Section IV to show the effectiveness of the main results; Finally, this paper is concluded in Section V.

The following notations will be used throughout this paper:

\textbf{Notations:} $\mathbb{R}^{n\times m}$ means the set of $n\times m$ real matrices, $\mathbb{R}^{n}$ indicates the $n$-dimensional Euclidean space, superscript $'$ denotes the transpose of a matrix; Real matrix $\mathcal{A}>0(\geq 0)$ means $\mathcal{A}$ is positive definite (positive semi-definite);  $Tr(Y)$ significants the trace of matrix $Y$, and $I_n$ means the $n \times n$ identity matrix.

\section{Optimal Control of Finite Horizon Case}
\subsection{Problem Formulation}
Consider the following discrete-time system:
\begin{equation}\label{sys1}
\left\{ \begin{array}{ll}
x_{k+1}^i=Ax_k^i+Bu_k^i,~~i=1,\cdots,v,\\
x_0^i=\xi^i,\\
\end{array} \right.
\end{equation}
where $x_k^i\in\mathbb{R}^n$ is the state of the $i$-th subsystem, $u_k^i\in\mathbb{R}^m$ is the corresponding control input, and $A\in\mathbb{R}^{n\times n}$, $B\in\mathbb{R}^{n\times m}$ are the given system matrices.

Denote the `average' state and `average' control as follows,
\begin{align}\label{ava}
  \bar{x}_k=\sum_{i=1}^{v}\mu_i x_k^i,~~ \bar{u}_k=\sum_{i=1}^{v}\mu_i u_k^i,
\end{align}where $\mu_i\in\mathbb{R}$ represents the weighting factor for the $i$-th system.

It can be easily derived from \eqref{sys1} that
\begin{align}\label{bar}
  \bar{x}_{k+1} & =A\bar{x}_k+B\bar{u}_k.
\end{align}

Associated with the $i$-th subsystem \eqref{sys1}, the cost function $J_N^i$ is introduced as:
\begin{align}\label{cos1}
  J_N^i &= \sum_{k=0}^{N}[(x_k^i)'Qx_k^i+(u_k^i)'Ru_k^i]
\end{align}
where $Q\in\mathbb{R}^{n\times n}$ and $R\in\mathbb{R}^{m\times m}$ are symmetric weighting matrices.

In this section, we consider the following imposed constraint on the `average' state:
\begin{align}\label{cons1}
  \bar{u}_k & =\bar{F}_k\bar{x}_k,
\end{align}
where $\bar{F}_k$ are given gain matrix for $k=0,\cdots,N$.

The main problem to be solved in this section can be formulated as:
\begin{problem}\label{PROB}
  For the given gain matrix $\bar{F}_k,k=0,\cdots,N$, to find $u_k^i$ to
\begin{align}\label{pro1}
  \text{minimize}~~ J_N=\sum_{i=1}^{v}J_N^i,
\end{align}
\begin{equation}\label{sys2}
\text{subject to:}\left\{ \begin{array}{ll}
x_{k+1}^i=Ax_k^i+Bu_k^i,~~i=1,\cdots,v,\\
\bar{u}_k =\bar{F}_k\bar{x}_k.
\end{array} \right.
\end{equation}
\end{problem}

\begin{remark}
  It should be pointed out that although the continuous-time version of Problem \ref{PROB} was studied in \cite{mm1,mm2}, Problem \ref{PROB} is still worth studying, the motivations are listed as follows. Firstly, the infinite horizon optimal distributed control and stabilization problems will be investigated, and we will develop the necessary and sufficient stabilization conditions, which have never been derived yet; Secondly, we study the discrete-time case in this paper; Thirdly, the main techniques in this paper differ from that in \cite{mm1,mm2}, we will develop the maximum principle for Problem \ref{PROB} and solve the associated FBDE.
\end{remark}

\subsection{Solution to Problem 1}
To develop the maximum principle (necessary condition) of minimizing $J_N$ in \eqref{pro1}, we define the Hamiltonian function as below:
\begin{align}\label{Ham}
  \mathcal{H}_k & =\frac{1}{2}\sum_{i=1}^{v}[(x_k^i)'Qx_k^i+(u_k^i)'Ru_k^i]\\
  &+\sum_{i=1}^{v}(p_k^i)'[(Ax_k^i+Bu_k^i)-x_{k+1}^i]
  +\hspace{-1mm}(p_k^{v+1})'(\bar{u}_k-\bar{F}_k\bar{x}_k),\notag
\end{align}
where $p_k^i$ denotes the costate for $i=1,\cdots,v+1$.

Then the following lemma (maximum principle) is introduced, which serves as the basic tool in this paper.
\begin{lemma}
The optimal distributed controller $u_k^i,i=1,\cdots,v$ satisfies the equilibrium equation:
\begin{align}
  \left(\frac{\partial \mathcal{H}_k}{\partial u_k^i}\right)' & =Ru_k^i+B'p_k^i+\mu_ip_k^{v+1}=0, k=0,\cdots,N.\label{equ1}
\end{align}

The costate $p_k^i$ obeys the following iteration (adjoint equation):
\begin{align}
  \left(\frac{\partial \mathcal{H}_k}{\partial x_k^i}\right)'\hspace{-1mm} & \hspace{-1mm} =\hspace{-1mm}p_{k-1}^i\hspace{-1mm}
  =\hspace{-1mm}Qx_k^i\hspace{-1mm}+\hspace{-1mm}A'p_k^i-\mu_i\bar{F}_k'p_k^{v+1}, 1\leq k\leq N,\label{equ2}
\end{align}
with final condition $p_N^i=0$ for $i=1,\cdots,v$.

\end{lemma}
\begin{proof}
  The proof can be found in \cite{lew}, which is omitted here.
\end{proof}

\begin{remark}
  The system dynamics \eqref{sys1} is forward, and the adjoint equation \eqref{equ2} is backward, \eqref{sys1} and \eqref{equ2} are called the system of FBDE. In the following, by decoupling the system of FBDE and using \eqref{equ1}, the optimal distributed control will be obtained.
\end{remark}

To guarantee the significance of Problem \ref{PROB}, we make the following standard assumption throughout this paper, see \cite{oka}.
\begin{assumption}\label{ass1}
  $Q\geq 0, R>0$.
\end{assumption}

The main results of this section can be concluded in the following theorem.
\begin{theorem}\label{th01}
For Problem \ref{PROB}, under Assumption \ref{ass1}, the optimal distributed controller of minimizing $\sum_{i=1}^{v}J_N^i$ can be presented as:
\begin{align}\label{oc1}
  u_k^i & =K_kx_k^i+\frac{\mu_i}{\sum_{i=1}^{v}\mu_i^2}\bar{K}_k\bar{x}_k,
\end{align}
where the gain matrices $K_k$ and $\bar{K}_k$ satisfy
\begin{align}
  K_k & =-(R+B'P_{k+1}B)^{-1}B'P_{k+1}A,\label{k} \\
  \bar{K}_k & =\bar{F}_k-K_k,\label{bark}
\end{align}
and $P_k$ obeys the following Riccati difference equations:
\begin{align}
  P_k & \hspace{-1mm}=\hspace{-1mm}Q\hspace{-1mm}+\hspace{-1mm}A'P_{k+1}A-A'P_{k+1}B(R+B'P_{k+1}B)^{-1}B'P_{k+1}A,\label{PP}
\end{align}
with final condition $P_{N+1}=0$.

In this case, the relationship between the costate $p_{k}^i$ and $x_{k+1}^i$, $\bar{x}_{k+1}$ (solution to the FBDE) can be described as:
\begin{align}
  p_{k}^i &=P_{k+1}x_{k+1}^i+\frac{\mu_i}{\sum_{i=1}^{v}\mu_i^2}\bar{P}_{k+1}\bar{x}_{k+1},i=1,\cdots,v,\label{sol1}\\
  p_k^{v+1}&=-\frac{1}{\sum_{i=1}^{v}\mu_i^2}[(R+B'P_{k+1}B+B'\bar{P}_{k+1}B)\bar{F}_k\notag\\
  &+B'(P_{k+1}+\bar{P}_{k+1})A]\bar{x}_k.\label{sol2}
\end{align}
where $\bar{P}_k$ satisfies the following relationship:
\begin{align}
 \bar{P}_k&=(A+B\bar{F}_k)'\bar{P}_{k+1}(A+B\bar{F}_k)+\bar{F}_k'(R+B'P_{k+1}B)\bar{F}_k\notag\\
  &+A'P_{k+1}B(R+B'P_{k+1}B)^{-1}B'P_{k+1}A\notag\\
  &+A'P_{k+1}B\bar{F}_k+\bar{F}_k'B'P_{k+1}A,\label{BARPP}
\end{align}
with final condition $\bar{P}_{N+1}=0$.

The optimal cost function is calculated as:
\begin{align}\label{ocf}
  J_N^* & =\sum_{i=1}^{v}[(x_0^i)'P_0x_0^i]+\frac{1}{\sum_{i=1}^{v}\mu_i^2}\bar{x}_0'\bar{P}_0\bar{x}_0,
\end{align}
where $P_0,\bar{P}_0$ can be derived from \eqref{PP} and \eqref{BARPP}.
\end{theorem}

\begin{proof}
We will show this theorem by using the induction methods.

Firstly, with $k=N$, by taking the weighting sum on both sides of \eqref{equ1}, we have
  \begin{align}\label{UN}
    R\bar{u}_N +B'\sum_{i=1}^{v}\mu_ip_N^i+\sum_{i=1}^{v}\mu_i^2p_{N}^{v+1}=0.
  \end{align}

Noting the constraint $\bar{u}_N=\bar{F}_N\bar{x}_N$ and $p_N^i=0$ for $1\leq i\leq v$, then there holds
\begin{align}\label{UN1}
  &R\bar{F}_N\bar{x}_N+\sum_{i=1}^{v}\mu_i^2p_{N}^{v+1}=0,
  \Rightarrow
p_{N}^{v+1}=-\frac{1}{\sum_{i=1}^{v}\mu_i^2} R\bar{F}_N\bar{x}_N.
\end{align}
Since $P_{N+1}=\bar{P}_{N+1}=0$, then \eqref{UN1} indicates \eqref{sol2} has been verified for $k=N$.

In what follows, by using $p_N^i=0$, it can be obtained from \eqref{equ1} that
\begin{align}\label{ui}
  0 & =Ru_{N}^i+B'p_N^i+\mu_ip_N^{v+1}=Ru_{N}^i-\frac{\mu_i}{\sum_{i=1}^{v}\mu_i^2} R\bar{F}_N\bar{x}_N,\notag\\
    &\Rightarrow u_N^i=\frac{\mu_i}{\sum_{i=1}^{v}\mu_i^2} \bar{F}_N\bar{x}_N.
\end{align}
i.e., \eqref{oc1} has been obtained for $k=N$.

Then following from \eqref{equ2} with $k=N$, there holds
\begin{align}\label{pN-1}
  p_{N-1}^i & =Qx_N^i+A'p_N^i-\mu_i\bar{F}_N'p_N^{v+1}\notag\\
  &=Qx_N^i+\frac{\mu_i}{\sum_{i=1}^{v}\mu_i^2}\bar{F}_N' R\bar{F}_N\bar{x}_N\notag\\
  &=P_Nx_N^i+\frac{\mu_i}{\sum_{i=1}^{v}\mu_i^2}\bar{P}_N\bar{x}_N,
\end{align}
where $P_{N+1}=\bar{P}_{N+1}=0$ has been used above, and $P_N,\bar{P}_N$ obey \eqref{PP}, \eqref{BARPP}, respectively.


Therefore, \eqref{oc1}-\eqref{BARPP} are verified for $k=N$.

To complete the induction, we assume \eqref{oc1}-\eqref{BARPP} are true for $k=l+1,\cdots,N$. i.e., for $k=l+1,\cdots,N$, we assume
\begin{itemize}
  \item The optimal distributed control $u_k^i$ satisfies \eqref{oc1};
  \item The costate $p_k^i, p_k^{v+1}$ satisfy \eqref{sol1}, \eqref{sol2}, respectively;
  \item The Riccati difference equations \eqref{PP} and \eqref{BARPP} hold.
\end{itemize}

Next we will show they are also true for $k=l$.

In fact, from \eqref{equ2}, $p_l^i$ can be calculated as:
\begin{align}\label{pl}
  p_{l}^i&=Qx_{l+1}^i+A'p_{l+1}^i-\mu_i\bar{F}_{l+1}'p_{l+1}^{v+1}\notag\\
&=Qx_{l+1}^i+A'[P_{l+2}x_{l+2}^i+\frac{\mu_i}{\sum_{i=1}^{v}\mu_i^2}\bar{P}_{l+2}\bar{x}_{l+2}]\notag\\
&+\frac{\mu_i}{\sum_{i=1}^{v}\mu_i^2}\bar{F}_{l+1}'[(R+B'P_{l+2}B+B'\bar{P}_{l+2}B)\bar{F}_{l+1}\notag\\
&+B'(P_{l+2}+\bar{P}_{l+2})A]\bar{x}_{l+1}\notag\\
&=(Q+A'P_{l+2}A)x_{l+1}^i+A'P_{l+2}Bu_{l+1}^i\notag\\
&+\frac{\mu_i}{\sum_{i=1}^{v}\mu_i^2}A'\bar{P}_{l+2}(A+B\bar{F}_{l+1})\bar{x}_{l+1}\notag\\
&+\frac{\mu_i}{\sum_{i=1}^{v}\mu_i^2}\bar{F}_{l+1}'[(R+B'P_{l+2}B+B'\bar{P}_{l+2}B)\bar{F}_{l+1}\notag\\
&+B'(P_{l+2}+\bar{P}_{l+2})A]\bar{x}_{l+1}\notag\\
&=[Q\hspace{-1mm}+\hspace{-1mm}A'P_{l+2}A\hspace{-1mm}-\hspace{-1mm}A'P_{l+2}B(R\hspace{-1mm}+\hspace{-1mm}B'P_{l+2}B)^{-1}B'P_{l+2}A]x_{l+1}^i\notag\\
&+\frac{\mu_i}{\sum_{i=1}^{v}\mu_i^2}\Big\{A'P_{l+2}B(R+B'P_{l+2}B)^{-1}B'P_{l+2}A\notag\\
&+A'P_{l+2}B\bar{F}_{l+1}+A'\bar{P}_{l+2}(A+B\bar{F}_{l+1})\notag\\
&+\bar{F}_{l+1}'(R+B'P_{l+2}B+B'\bar{P}_{l+2}B)\bar{F}_{l+1}\notag\\
&+\bar{F}_{l+1}'B'(P_{l+2}+\bar{P}_{l+2})A\Big\}\bar{x}_{l+1}\notag\\
&=P_{l+1}x_{l+1}^i+\frac{\mu_i}{\sum_{i=1}^{v}\mu_i^2}\bar{P}_{l+1}\bar{x}_{l+1},
\end{align}
where \eqref{oc1}, \eqref{PP} and \eqref{BARPP} have been used for $k=l+1$.

Taking the weighted summation on \eqref{equ1} for $k=l$, there holds
\begin{align}\label{ul}
  R\bar{u}_l+B'\sum_{i=1}^{v}\mu_ip_l^i+\sum_{i=1}^{v}\mu_i^2p_{l}^{v+1}=0,
\end{align}
The constraint \eqref{cons1} reads that $\bar{u}_l=\bar{F}_l\bar{x}_l$, then \eqref{ul} indicates that
\begin{align}\label{ind1}
&R\bar{F}_l\bar{x}_l+\sum_{i=1}^{v}B'(P_{l+1}\mu_ix_{l+1}^i+\frac{\mu_i^2}{\sum_{i=1}^{v}\mu_i^2}\bar{P}_{l+1}\bar{x}_{l+1})\notag\\
&+\sum_{i=1}^{v}\mu_i^2p_{l}^{v+1}=0,\notag\\
\Rightarrow~~&R\bar{F}_l\bar{x}_l+B'(P_{l+1}\bar{x}_{l+1}+\bar{P}_{l+1}\bar{x}_{l+1})+\sum_{i=1}^{v}\mu_i^2p_{l}^{v+1}=0,\notag\\
\Rightarrow~~&p_l^{v+1}=-\frac{1}{\sum_{i=1}^{v}\mu_i^2}[(R+B'P_{l+1}B+B'\bar{P}_{l+1}B)\bar{F}_l\notag\\
&+B'(P_{l+1}+\bar{P}_{l+1})A]\bar{x}_l.
\end{align}

Thus \eqref{pl} and \eqref{ind1} yield that \eqref{sol1}-\eqref{sol2} are true for $k=l$.

In what follows, from \eqref{equ1}, by using \eqref{pl} and \eqref{ind1} we have
\begin{align}\label{ul2}
0&=Ru_l^i+B'p_l^i+\mu_ip_l^{v+1}\notag\\
&=Ru_l^i+B'(P_{l+1}x_{l+1}^i+\frac{\mu_i}{\sum_{i=1}^{v}\mu_i^2}\bar{P}_{l+1}\bar{x}_{l+1})\notag\\
&-\frac{\mu_i}{\sum_{i=1}^{v}\mu_i^2}[(R+B'P_{l+1}B+B'\bar{P}_{l+1}B)\bar{F}_l\notag\\
&+B'(P_{l+1}+\bar{P}_{l+1})A]\bar{x}_l\notag\\
&=
(R+B'P_{l+1}B)u_l^i+B'P_{l+1}Ax_l^i\notag\\
&+\frac{\mu_i}{\sum_{i=1}^{v}\mu_i^2}\Big\{B'P_{l+1}(A+B\bar{F}_l)\notag\\
&-(R+B'P_{l+1}B+B'\bar{P}_{l+1}B)\bar{F}_l\hspace{-1mm}-\hspace{-1mm}B'(P_{l+1}\hspace{-1mm}+\hspace{-1mm}\bar{P}_{l+1})A\Big\}\bar{x}_l\notag\\
&=(R+B'P_{l+1}B)u_l^i+B'P_{l+1}Ax_l^i\notag\\
&-\frac{\mu_i}{\sum_{i=1}^{v}\mu_i^2}\Big\{(R+B'P_{l+1}B)\bar{F}_l+B'P_{l+1}A\Big\}\bar{x}_l,\notag\\
\Rightarrow~~&u_l^i=K_lx_l^i+\frac{\mu_i}{\sum_{i=1}^{v}\mu_i^2}\bar{K}_l\bar{x}_l,
\end{align}
where $K_l,\bar{K}_l$ are given by \eqref{k} and \eqref{bark}. Then \eqref{oc1} is verified for $k=l$.

Therefore, \eqref{oc1}-\eqref{BARPP} are true for $k=l$, this completes the induction methods, and \eqref{oc1}-\eqref{BARPP} can be developed for any $k=0,\cdots,N$.

Finally, we will calculate the optimal cost function. For simplicity, we denote
\begin{align}\label{VN}
  V_N(k,x_k,\bar{x}_k)=\sum_{i=1}^{v}(x_k^i)'P_kx_k^i+\frac{1}{\sum_{i=1}^{v}\mu_i^2}\bar{x}_k'\bar{P}_k\bar{x}_k.
\end{align}

It can be verified from \eqref{VN} that
\begin{align}\label{costa}
 &~~~ V_N(k+1,x_{k+1},\bar{x}_{k+1}) \notag\\ &=\sum_{i=1}^{v}(x_{k+1}^i)'P_{k+1}x_{k+1}^i+\frac{1}{\sum_{i=1}^{v}\mu_i^2}\bar{x}_{k+1}'\bar{P}_{k+1}\bar{x}_{k+1}\notag\\
  &=\sum_{i=1}^{v}(Ax_{k}^i+Bu_k^i)'P_{k+1}(Ax_{k}^i+Bu_k^i)\notag\\
  &+\frac{1}{\sum_{i=1}^{v}\mu_i^2}[(A+B\bar{F}_k)\bar{x}_{k}]'\bar{P}_{k+1}[(A+B\bar{F}_k)\bar{x}_{k}]\notag\\
  &=\sum_{i=1}^{v}\Big[(x_k^i)'A'P_{k+1}Ax_k^i+2(u_k^i)'B'P_{k+1}Ax_k^i\notag\\
  &~~~+(u_k^i)'B'P_{k+1}Bu_k^i\Big]\notag\\
  &+\frac{1}{\sum_{i=1}^{v}\mu_i^2}\bar{x}_k'(A+B\bar{F}_k)'\bar{P}_{k+1}(A+B\bar{F}_k)\bar{x}_k.
\end{align}

Noting the fact that
\begin{align}\label{alij}
  &\sum_{i=1}^{v}(u_k^i-K_kx_k^i-\frac{\mu_i}{\sum_{i=1}^{v}\mu_i^2}\bar{K}_k\bar{x}_k)'(R+B'P_{k+1}B)\\
  &\times(u_k^i-K_kx_k^i-\frac{\mu_i}{\sum_{i=1}^{v}\mu_i^2}\bar{K}_k\bar{x}_k)\notag\\
=& \sum_{i=1}^{v}\Big[(u_k^i)'(R+B'P_{k+1}B)u_k^i+2(u_k^i)'B'P_{k+1}Ax_k^i\notag\\
&+(x_k^i)'A'P_{k+1}B(R+B'P_{k+1}B)^{-1}B'P_{k+1}Ax_k^i\Big]\notag\\
&-\frac{1}{\sum_{i=1}^{v}\mu_i^2}\bar{x}_k'\Big[\bar{F}_k'(R+B'P_{k+1}B)\bar{F}_k+A'P_{k+1}B\bar{F}_k\notag\\
&+\bar{F}_k'B'P_{k+1}A\hspace{-1mm}+\hspace{-1mm}A'P_{k+1}B(R\hspace{-1mm}+\hspace{-1mm}B'P_{k+1}B)^{-1}B'P_{k+1}A\Big]\bar{x}_k.\notag
\end{align}

Then, combining \eqref{VN}-\eqref{alij} yields that
\begin{align}\label{minu}
  &~~~V_N(k,x_{k},\bar{x}_{k})-V_N(k+1,x_{k+1},\bar{x}_{k+1})\notag\\
  &=\sum_{i=1}^{v}\Big\{(x_k^i)'[P_k-A'P_{k+1}A\notag\\
  &+A'P_{k+1}B(R+B'P_{k+1}B)^{-1}B'P_{k+1}A]x_k^i+(u_k^i)'Ru_k^i\Big\}\notag\\
  &+\frac{1}{\sum_{i=1}^{v}\mu_i^2}\bar{x}_k'\Big[\bar{P}_k-(A+B\bar{F}_k)'\bar{P}_{k+1}(A+B\bar{F}_k)\notag\\
  &-\bar{F}_k'(R+B'P_{k+1}B)\bar{F}_k\notag\\
  &-A'P_{k+1}B\bar{F}_k-\bar{F}_k'B'P_{k+1}A\notag\\
  &-A'P_{k+1}B(R+B'P_{k+1}B)^{-1}B'P_{k+1}A\Big]\bar{x}_k\notag\\
  &-\sum_{i=1}^{v}(u_k^i-K_kx_k^i-\frac{\mu_i}{\sum_{i=1}^{v}\mu_i^2}\bar{K}_k\bar{x}_k)'(R+B'P_{k+1}B)\notag\\
  &\times (u_k^i-K_kx_k^i-\frac{\mu_i}{\sum_{i=1}^{v}\mu_i^2}\bar{K}_k\bar{x}_k)\notag\\
  &=\sum_{i=1}^{v}[(x_{k}^i)'Qx_k^i+(u_k^i)'Ru_k^i]\notag\\
  &-\sum_{i=1}^{v}(u_k^i-K_kx_k^i-\frac{\mu_i}{\sum_{i=1}^{v}\mu_i^2}\bar{K}_k\bar{x}_k)'(R+B'P_{k+1}B)\notag\\
  &\times(u_k^i-K_kx_k^i-\frac{\mu_i}{\sum_{i=1}^{v}\mu_i^2}\bar{K}_k\bar{x}_k),
\end{align}
where $K_k,\bar{K}_k$ are as in \eqref{k}, \eqref{bark}.

Taking summation on both sides of \eqref{minu} from $0$ to $N$ and noting $P_{N+1}=\bar{P}_{N+1}=0$, we have
\begin{align}\label{minu2}
 &\sum_{i=1}^{v}(x_0^i)'P_0x_0^i+\frac{1}{\sum_{i=1}^{v}\mu_i^2}\bar{x}_0'\bar{P}_0\bar{x}_0\notag\\
 &=\sum_{i=1}^{v}\sum_{k=0}^{N}[(x_{k}^i)'Qx_k^i+(u_k^i)'Ru_k^i]\notag\\
 &-\sum_{i=1}^{v}(u_k^i-K_kx_k^i-\frac{\mu_i}{\sum_{i=1}^{v}\mu_i^2}\bar{K}_k\bar{x}_k)'(R+B'P_{k+1}B)\notag\\
  &\times(u_k^i-K_kx_k^i-\frac{\mu_i}{\sum_{i=1}^{v}\mu_i^2}\bar{K}_k\bar{x}_k).
\end{align}

On the other hand, from \eqref{PP} we know
\begin{align}\label{geh}
  P_k & \hspace{-1mm}=\hspace{-1mm}Q\hspace{-1mm}+\hspace{-1mm}A'P_{k+1}A-A'P_{k+1}B(R+B'P_{k+1}B)^{-1}B'P_{k+1}A\notag\\
  &=Q+K_k'RK_k+(A+BK_k)'P_{k+1}(A+BK_k),
\end{align}
where $K_k$ satisfies \eqref{k}.

Again by using the induction method, \eqref{geh} implies that $P_k\geq 0$ for $k\geq 0$, then $R+B'P_{k+1}B>0$ can be easily obtained.

Next, there holds from \eqref{geh} that
\begin{align}\label{zll}
  J_N &=\sum_{i=1}^{v}\sum_{k=0}^{N}[(x_{k}^i)'Qx_k^i+(u_k^i)'Ru_k^i]\notag\\
  &=\sum_{i=1}^{v}(x_0^i)'P_0x_0^i+\frac{1}{\sum_{i=1}^{v}\mu_i^2}\bar{x}_0'\bar{P}_0\bar{x}_0\notag\\
 &+\sum_{i=1}^{v}(u_k^i-K_kx_k^i-\frac{\mu_i}{\sum_{i=1}^{v}\mu_i^2}\bar{K}_k\bar{x}_k)'(R+B'P_{k+1}B)\notag\\
  &\times(u_k^i-K_kx_k^i-\frac{\mu_i}{\sum_{i=1}^{v}\mu_i^2}\bar{K}_k\bar{x}_k)\notag\\
  &\geq \sum_{i=1}^{v}(x_0^i)'P_0x_0^i+\frac{1}{\sum_{i=1}^{v}\mu_i^2}\bar{x}_0'\bar{P}_0\bar{x}_0.
\end{align}

Therefore, the optimal cost function can be given from \eqref{minu2} as \eqref{zll}, and the optimal controller is as \eqref{oc1}. This ends the proof.
\end{proof}

\begin{corollary}\label{coro1}
  If the distributed controller is chosen to be $u_k^i=\bar{F}_kx_k^i$, then the constraint \eqref{cons1} is satisfied. Moreover, the cost function $J_N^i$ in \eqref{cos1} can be presented as:
  \begin{align}\label{jnii}
    \tilde{J}_N^i & =(x_0^i)'(P_0+\bar{P}_0)x_0^i
  \end{align}
  where $P_0+\bar{P}_0$ can be obtained from the following Lyapunov equation:
  \begin{align}\label{leqa}
    P_k+\bar{P}_k & =Q+\bar{F}_k'R\bar{F}_k\notag\\
    &~~+(A+B\bar{F}_k)'(P_{k+1}+\bar{P}_{k+1})(A+B\bar{F}_k),
  \end{align}
  and $P_k$, $\bar{P}_k$ obey \eqref{PP}, \eqref{BARPP}, respectively.
\end{corollary}
\begin{proof}
  The results can be easily verified, and thus the proof is omitted here.
\end{proof}

\begin{remark}\label{rem1}
  For the $i$-th individual system \eqref{sys1}, the optimal controller $u_k^i$ of minimizing $J_N^i$ in \eqref{cos1} can be given as
  \begin{align}\label{oc1a}
    u_k^i &=K_kx_k^i,
  \end{align}
where $K_k$ satisfies \eqref{k}, and the minimizing cost function $J_N^{i,*}$ is presented as
  \begin{align}\label{jni8}
    J_N^{i,*} & =(x_i^0)'P_0x_i^0,
  \end{align}
  and $P_0$ can be derived from Riccati equation \eqref{PP}. This is actually the discrete-time standard LQ control problem, see \cite{lew, oka}.
\end{remark}

\begin{remark}\label{rem2}
  It is noted from \eqref{jnii} and \eqref{jni8} that for any initial state $x_0^i$, there must hold
  \begin{align}\label{compaaas}
    (x_i^0)'P_0x_i^0= J_N^{i,*} & \leq \tilde{J}_N^i =(x_0^i)'(P_0+\bar{P}_0)x_0^i,\notag\\
    & \Rightarrow~P_0\leq P_0+\bar{P}_0,
  \end{align}
  then the positive semi-definite of $\bar{P}_0$ can be obtained, i.e., $\bar{P}_0\geq 0$.

  On the other hand, from \eqref{ocf} we know that $J_N^*\geq \sum_{i=1}^{v}J_N^{i,*}$, and $\bar{x}_0'\bar{P}_0\bar{x}_0$ represents the \textbf{coordination term} subject to the constraint \eqref{cons1}.
\end{remark}

\begin{remark}\label{rem3}
  With $\bar{F}_k=K_k$, it can be easily verified that $J_N^*=\sum_{i=1}^{v}J_N^{i,*}$. This implies that if the given matrices $\bar{F}_k$ happens to be chosen as the optimal gain matrix $K_k$, then the optimal $J_N^*$ in \eqref{ocf} equals to the sum of the optimal $J_N^{*,i}$ in \eqref{jni8}.
\end{remark}

\begin{remark}
  For the $i$-th subsystem, the optimal distributed controller \eqref{oc1} only relies on the state of itself $x_k^i$ and the `average' state $\bar{x}_k$. On the other hand, the computation of distributed controller \eqref{oc1} is actually off-line, and the computation burden is independent of the number of subsystems.
\end{remark}

\section{Stabilization and Control of Infinite Horizon Case}
\subsection{Problem Formulation}
Associated with system \eqref{sys1}, the corresponding cost function of the infinite horizon is introduced as:
\begin{align}\label{cost2}
J=\sum_{i=1}^{v}\sum_{k=0}^{\infty}[(x_k^i)'Qx_k^i+(u_k^i)'Ru_k^i].
\end{align}

For the given matrix $\bar{F}$, the constraint on the `average' state $\bar{x}_k$ for the infinite horizon case can be described as
\begin{align}\label{stran1}
  \bar{u}_k & =\bar{F}\bar{x}_k.
\end{align}

In this section, we will consider the stabilization and control problems for system \eqref{sys1} and \eqref{cost2}.

\begin{definition}\label{def1}
  System \eqref{sys1} is called asymptotically stable with $u_k^i=0$, if for any initial state $x_0^i$, there holds $\lim_{k\rightarrow +\infty}x_k^i=0$.
\end{definition}

\begin{definition}\label{def2}
  System \eqref{sys1} is called stabilizable if there exists a linear distributed controller $u_k^i$ in terms of $x_k^i$ and $\bar{x}_k$ such that the closed-loop system \eqref{sys1} is asymptotically stable.
\end{definition}

The problem under consideration in this section can be described as follows.
\begin{problem}\label{pro2}
 Subject to \eqref{stran1}, seek for a distributed controller $u_k^i, i=1,\cdots,v$ such that the system \eqref{sys1} is stabilizable and the cost function \eqref{cost2} is minimized.
\end{problem}

Moreover, to handle the stabilization problem, the standard assumption is made in this section:
\begin{assumption}\label{aobs}
  $(A,C)$ is observable, where $C'C=Q$.
\end{assumption}

Before we present the main conclusion, we will develop some useful lemmas as below.
\begin{lemma}\label{lem02}
  Under Assumptions \ref{ass1} and \ref{aobs}, there exists $N_0>0$ such that $P_0(N)+\bar{P}_0(N)\geq P_0(N)>0$ for any $N>N_0$.
\end{lemma}
\begin{proof}
As stated in Remark \ref{rem2}, $P_0(N)+\bar{P}_0(N)\geq P_0(N)$ for any $N$, then we only need to show there exists $N_0$ such that $P_0(N)>0$ for $N>N_0$.

Otherwise, for arbitrary $N$, there always exists nonzero $x\in\mathbb{R}^n$ such that $x'P_0(N)x=0$.

From Remark \ref{rem1}, by letting the initial state be $x$ defined above, we can obtain
\begin{align}\label{posd}
  J_N^{i,*}& =\sum_{k=0}^{N}[(x_k^i)'Qx_k^i+(u_k^i)'Ru_k^i]=x'P_0(N)x=0.
\end{align}
By using Assumption \ref{ass1} we know $Q\geq 0$ and $R>0$, then \eqref{posd} indicates that
\begin{align}\label{posd2}
u_k^i=0, ~\text{and}~~Cx_k^i=0,~~k=0,\cdots,N.
\end{align}

Next by using the observable of $(A,C)$ as in Assumption \ref{aobs}, we can conclude that $x_0=x=0$, which contradicts with the nonzero of $x$ defined above. This completes the proof.
\end{proof}

\begin{lemma}\label{thm2}
  Under Assumptions \ref{ass1} and \ref{aobs}, suppose the system \eqref{sys1} is stabilizable, then the following coupled AREs have solution $P, \bar{P}$ satisfying $P+\bar{P}\geq P>0$:
\begin{align}
  P & =Q+A'PA-A'PB(R+B'PB)^{-1}B'PA,\label{are1}\\
  \bar{P}&=(A+B\bar{F})'\bar{P}(A+B\bar{F})+\bar{F}'(R+B'PB)\bar{F}\notag\\
  &+A'PB(R+B'PB)^{-1}B'PA+A'PB\bar{F}+\bar{F}'B'PA. \label{are2}
\end{align}
  \end{lemma}
\begin{proof}
   Firstly, the monotonically increasing of $P_0(N)$ and $P_0(N)+\bar{P}_0(N)$ with respect to $N$ will be shown as below.

  In fact, recall from Remark \ref{rem1} and Corollary \ref{coro1}, the following two situations are considered:

  1) From Remark \ref{rem1} we know that the minimizing $J_N^{i,*}$ can be presented as:
  \begin{align}\label{oc01}
    J_N^{i,*} & =(x_0^i)'P_0(N)x_0^i.
  \end{align}

  Noting the fact that $J_N^i\leq J_{N+1}^i$, then $J_N^{i,*}\leq J_{N+1}^{i,*}$, from \eqref{oc01} we have
  $$P_0(N)\leq P_{0}(N+1).$$

  2) With $u_k^i=\bar{F}_k(N)x_k^i$ and using Corollary \ref{coro1}, it can be obtained,
  \begin{align}\label{oc02}
   &(x_0^i)'[P_0(N)+\bar{P}_0(N)]x_0^i =\tilde{J}_N^i\notag\\
   \leq &\tilde{J}_{N+1}^i=(x_0^i)'[P_0(N+1)+\bar{P}_0(N+1)]x_0^i.
  \end{align}

  It can be implied from \eqref{oc02} that
  $$P_0(N)+\bar{P}_0(N)\leq P_0(N+1)+\bar{P}_0(N+1).$$

  Next we will prove the boundedness of $P_0(N)$ and $P_0(N)+\bar{P}_0(N)$.

  In fact, since system \eqref{sys1} is stabilizable for $i=1,\cdots,v$, from Definition \ref{def2} we know that there exists constant matrix $L$ such that $u_k^i=Lx_k^i+\bar{L}\bar{x}_k$ stabilizes system \eqref{sys1}, i.e., $\lim_{k\rightarrow+\infty}x_k^i=0,~i=1,\cdots,v$.

   On the other hand, from \eqref{ava} $\bar{x}_k=\sum_{i=1}^{v}\mu_i x_k^i$, then it can be easily obtained that $\lim_{k\rightarrow+\infty}\bar{x}_k=0$.

 Furthermore, by using \eqref{ava}, \eqref{stran1} and the results of references \cite{lsts,mxzhou}, we can conclude
  \begin{align}\label{oc03}
   \sum_{k=0}^{\infty} \bar{x}_k<+\infty,~ \sum_{i=1}^{v}\sum_{k=0}^{\infty} x_k^i<+\infty, ~i=1,\cdots,v.
  \end{align}

For constant matrices $L,\bar{L}$, there exists constant $\lambda$ satisfying:
$$Q+L'RL\leq \lambda I,~ \bar{L}'R\bar{L}\leq \lambda I,~\bar{L}'RL L'R\bar{L} \leq \lambda I.$$

 Thus, by using Schwarz inequality,  from \eqref{oc03} we know the following relationship holds
  \begin{align}\label{rela1}
  &\sum_{i=1}^{v}[(x_0^i)'P_0(N)x_0^i]+\frac{1}{\sum_{i=1}^{v}\mu_i^2}\bar{x}_0'\bar{P}_0(N)\bar{x}_0=J_N^*\notag\\
  &\leq J=\sum_{i=1}^{v}\sum_{k=0}^{\infty}[(x_k^i)'Qx_k^i+(u_k^i)'Ru_k^i]\notag\\
  &=\sum_{i=1}^{v}\sum_{k=0}^{\infty}[(x_k^i)'(Q+L'RL)x_k^i+\bar{x}_k'\bar{L}'R\bar{L}\bar{x}_k\notag\\
  &+2(x_k^i)'L'R\bar{L}\bar{x}_k]\notag\\
  &\leq \lambda \sum_{i=1}^{v}\sum_{k=0}^{\infty}[(x_k^i)'x_k^i+ \bar{x}_k'\bar{x}_k]\notag\\
  &+\sum_{i=1}^{v}\sum_{k=0}^{\infty}[(x_k^i)'x_k^i+ \bar{x}_k'\bar{L}'RLL'R\bar{L}\bar{x}_k]\notag\\
  &\leq (\lambda+1)\sum_{i=1}^{v}\sum_{k=0}^{\infty}(x_k^i)'x_k^i+ 2\lambda\sum_{i=1}^{v}\sum_{k=0}^{\infty}\bar{x}_k'\bar{x}_k<+\infty.
  \end{align}

  From \eqref{geh} and Remark \ref{rem2}, we have $P_0(N)\geq0,\bar{P}_0(N)\geq 0$, then $P_0(N)$ and $P_0(N)+\bar{P}_0(N)$ are bounded.

   Recall that both $P_0(N)$ and $P_0(N)+\bar{P}_0(N)$ are increasing with $N$, thus $P_0(N)$ and $P_0(N)+\bar{P}_0(N)$ are convergent with $N$, i.e., there exists $P$ and $\bar{P}$  satisfying
  \begin{align}\label{conve}
   P=\lim_{N\rightarrow+\infty}P_0(N),~~\bar{P}=\lim_{N\rightarrow+\infty}\bar{P}_0(N).
  \end{align}

 Finally, by using Lemma \ref{lem02} and relationship \eqref{conve}, the positive definiteness of $P$ and $P+\bar{P}$ can be obtained immediately. This completes the proof.\end{proof}

\subsection{Solution to Problem \ref{pro2}}

In the following theorem, we will explore the necessary and sufficient stabilization conditions for system \eqref{sys1}.
\begin{theorem}\label{thsta}
 Suppose Assumptions \ref{ass1} and \ref{aobs} hold, then the following assertions are equivalent:

 1) The system \eqref{sys1} is stabilizable;

 2) $A+B\bar{F}$ is stable;

 3) The coupled AREs \eqref{are1}-\eqref{are2} have unique positive definite solution ($P+\bar{P}\geq P>0$).

In this case, the stabilizing controller is given by
\begin{align}\label{stc}
  u_k^i & =Kx_k^i+\frac{\mu_i}{\sum_{i=1}^{v}\mu_i^2}\bar{K}\bar{x}_k,
\end{align}
where $K,\bar{K}$ are as:
\begin{align}\label{kkk}
  K & =-(R+B'PB)^{-1}B'PA,~\bar{K}=\bar{F}-K.
\end{align}

Moreover, the stabilizing controller \eqref{stc} also minimizes the cost function \eqref{cost2}, and the minimizing $J^*$ can be presented as:
\begin{align}\label{mas}
  J^*=\sum_{i=1}^{v}(x_0^i)'Px_0^i+\frac{1}{\sum_{i=1}^{v}\mu_i^2}\bar{x}_0'\bar{P}\bar{x}_0.
\end{align}
\end{theorem}

\begin{proof}

Firstly, we will show the equivalence of 1) and 2).

Actually, if the system \eqref{sys1} is stabilizable, i.e., $\lim_{k\rightarrow \infty}x_k^i= 0$ for $i=1,\cdots,v$, then from the definition of $\bar{x}_k=\sum_{i=1}^{v}\mu_i x_k^i$ given by \eqref{ava}, we know the stabilization of $x_k^i$ is equivalent to the stabilization of $\bar{x}_k$.

On the other hand, combining \eqref{ava} with \eqref{stran1} yields $\bar{x}_{k+1}=(A+B\bar{F})x_k$, which indicates that the stabilization of $\bar{x}_k$ is equivalent to the stable of $A+B\bar{F}$. Therefore, the equivalence of 1) and 2) has been shown.

Next, we will prove ``1) $\Leftrightarrow$ 3)".

  `3) $\Rightarrow$ 1)': Under Assumptions \ref{ass1}-\ref{aobs}, suppose the AREs \eqref{are1}-\eqref{are2} admit unique positive definite solution with $P+\bar{P}\geq P>0$, we will show that system \eqref{sys1} can be stabilized with controller \eqref{stc}.

Firstly, we define the following Lyapunov function as follows,
\begin{align}\label{lf}
V(k,x_k)=\sum_{i=1}^{v}(x_k^i)'Px_k^i+\frac{1}{\sum_{i=1}^{v}\mu_i^2}\bar{x}_k'\bar{P}\bar{x}_k.
\end{align}

Similar to the derivation of \eqref{costa}-\eqref{minu}, we can obtain
\begin{align}\label{lf2}
 &~~ V(k,x_k)-V(k+1,x_{k+1})\notag\\
  &=\sum_{i=1}^{v}\Big\{(x_k^i)'[P-A'PA+A'PB(R+B'PB)^{-1}B'PA]x_k^i\notag\\
  &+(u_k^i)'Ru_k^i\Big\}\notag\\
  &+\frac{1}{\sum_{i=1}^{v}\mu_i^2}\bar{x}_k'\Big[\bar{P}\hspace{-1mm}-\hspace{-1mm}(A\hspace{-1mm}+\hspace{-1mm}B\bar{F})'
  \bar{P}(A\hspace{-1mm}+\hspace{-1mm}B\bar{F})
  \hspace{-1mm}-\hspace{-1mm}\bar{F}'(R+B'PB)\bar{F}\notag\\
  &-A'PB\bar{F}-\bar{F}'B'PA-A'PB(R+B'PB)^{-1}B'PA\Big]\bar{x}_k\notag\\
  &-\sum_{i=1}^{v}(u_k^i-Kx_k^i-\frac{\mu_i}{\sum_{i=1}^{v}\mu_i^2}\bar{K}\bar{x}_k)'(R+B'PB)\notag\\&\times
  (u_k^i-Kx_k^i-\frac{\mu_i}{\sum_{i=1}^{v}\mu_i^2}\bar{K}\bar{x}_k)\notag\\
  &=\sum_{i=1}^{v}[(x_{k}^i)'Qx_k^i+(u_k^i)'Ru_k^i]\notag\\
  &-\sum_{i=1}^{v}(u_k^i-Kx_k^i-\frac{\mu_i}{\sum_{i=1}^{v}\mu_i^2}\bar{K}\bar{x}_k)'(R+B'PB)\notag\\&\times
  (u_k^i-Kx_k^i-\frac{\mu_i}{\sum_{i=1}^{v}\mu_i^2}\bar{K}\bar{x}_k)\notag\\ &=\sum_{i=1}^{v}[(x_{k}^i)'Qx_k^i+(u_k^i)'Ru_k^i]\geq 0.
\end{align}

On the other hand, since $P> 0$ and $\bar{P}\geq 0$, we know that $V(k,x_k)$ is bounded below. Thus the convergence of $V(k,x_k)$ is derived.

Since the coefficient matrices $A, B$ are time-invariant, then via a time-shift of $m$ and taking summation on both sides of \eqref{lf2}, there holds
\begin{align}\label{tshf}
  &\lim_{m\rightarrow+\infty}V(m,x_m)-V(m+N+1,x_{m+N+1})\notag\\
  &=\lim_{m\rightarrow+\infty}\sum_{i=1}^{v}\sum_{k=m}^{m+N}[(x_{k}^i)'Qx_k^i+(u_k^i)'Ru_k^i]=0,
\end{align}
where the convergence of $V(k,x_k)$ has been inserted.

Furthermore, from Theorem \ref{th01} we know that
\begin{align}\label{tshf2}
 0 &=\lim_{m\rightarrow+\infty}\sum_{i=1}^{v}\sum_{k=m}^{m+N}[(x_{k}^i)'Qx_k^i+(u_k^i)'Ru_k^i]\notag\\
  &\geq \lim_{m\rightarrow+\infty}\sum_{i=1}^{v}[(x_m^i)'P_m(m+N)x_m^i]\notag\\
  &+\frac{1}{\sum_{i=1}^{v}\mu_i^2}\bar{x}_m'\bar{P}_m(m+N)\bar{x}_m\notag\\
  &= \lim_{m\rightarrow+\infty}\sum_{i=1}^{v}[(x_m^i)'P_0(N)x_m^i]+\frac{1}{\sum_{i=1}^{v}\mu_i^2}\bar{x}_m'\bar{P}_0(N)\bar{x}_m.
\end{align}

Since $\bar{P}_0(N)\geq 0$, then \eqref{tshf2} indicates that
$$\lim_{m\rightarrow+\infty}[(x_m^i)'P_0(N)x_m^i]=0.$$

By using Lemma \ref{lem02} the stabilization of system \eqref{sys1} can be derived, i.e., $\lim_{m\rightarrow+\infty}x_m^i=0$.

Finally, we will show the stabilizing controller \eqref{stc} minimizes the cost function \eqref{cost2}.

In fact, taking summation on both sides of \eqref{lf2} from $0$ to $N$, and taking limitation of $N\rightarrow+\infty$, we have
\begin{align}\label{subn}
  V(0,x_0)&=-\sum_{i=1}^{v}\sum_{k=0}^{\infty}(u_k^i\hspace{-1mm}
  -\hspace{-1mm}Kx_k^i\hspace{-1mm}-\hspace{-1mm}
  \frac{\mu_i}{\sum_{i=1}^{v}\mu_i^2}\bar{K}\bar{x}_k)'(R\hspace{-1mm}+\hspace{-1mm}B'PB)\notag\\
  &\times (u_k^i-Kx_k^i-\frac{\mu_i}{\sum_{i=1}^{v}\mu_i^2}\bar{K}\bar{x}_k)\notag\\
  &+\sum_{i=1}^{v}\sum_{k=0}^{\infty}[(x_{k}^i)'Qx_k^i+(u_k^i)'Ru_k^i],
  \end{align}
where $\lim_{N\rightarrow+\infty}V(N,x_N)=0$ has been inserted.

Thus \eqref{cost2} yields
\begin{align}\label{JJ}
  J & =\sum_{i=1}^{v}\sum_{k=0}^{\infty}(x_k^i)'Px_k^i+\frac{1}{\sum_{i=1}^{v}\mu_i^2}\bar{x}_k'\bar{P}\bar{x}_k\notag\\
  &+\sum_{i=1}^{v}\sum_{k=0}^{\infty}(u_k^i-Kx_k^i-\frac{\mu_i}{\sum_{i=1}^{v}\mu_i^2}\bar{K}\bar{x}_k)'(R+B'PB)\notag\\
  &\times (u_k^i-Kx_k^i-\frac{\mu_i}{\sum_{i=1}^{v}\mu_i^2}\bar{K}\bar{x}_k).
\end{align}

Since $R>0$ in Assumption \ref{ass1} and $P>0$ (i.e., $R+B'PB>0$), thus obviously from \eqref{JJ} we know that the stabilizing controller \eqref{stc} minimizes the cost function \eqref{cost2}, and the minimizing cost function is given by \eqref{mas}.

  `1) $\Rightarrow$ 3)': Under Assumptions \ref{ass1}-\ref{aobs}, if system \eqref{sys1} is stabilizable, we will show the coupled AREs \eqref{are1}-\eqref{are2} have unique positive definite solution $P,\bar{P}$ satisfying $P+\bar{P}\geq P>0$.

  From Lemma \ref{thm2}, we know that AREs \eqref{are1}-\eqref{are2} have positive definite solution such that $P+\bar{P}\geq P>0$, what remains to show is the uniqueness of $P$ and $\bar{P}$.

  In fact, if there exist another $S,\bar{S}$ ($S+\bar{S}\geq S>0$) satisfying
\begin{align}
  S & =Q+A'SA-A'SB(R+B'SB)^{-1}B'SA,\label{are3}\\
  \bar{S}&=(A+B\bar{F})'\bar{S}(A+B\bar{F})+\bar{F}'(R+B'SB)\bar{F}\notag\\
  &+A'SB(R+B'SB)^{-1}B'SA+A'SB\bar{F}+\bar{F}'B'SA. \label{are4}
\end{align}

1) Similar to Remark \ref{rem1}, for any $x_0^i$, we know that the minimizing $J^{i,*}$ can be presented as:
  \begin{align}\label{soc01}
    J^{i,*} & =(x_0^i)'Px_0^i=(x_0^i)'Sx_0^i.
  \end{align}
Thus $P=S$ can be derived from \eqref{soc01}.

  2) On the other hand, similar to Corollary \ref{coro1}, with $u_k^i=\bar{F}x_k^i$, then for any $x_0^i$ there holds
  \begin{align}\label{soc02}
   &\tilde{J}^i=(x_0^i)'(P+\bar{P})x_0^i =(x_0^i)'(S+\bar{S})x_0^i.
  \end{align}
  Therefore, we can conclude $P+\bar{P}=S+\bar{S}$. The uniqueness of $P,\bar{P}$ has been proven. The proof is complete.
\end{proof}
\begin{remark}
  The infinite horizon stabilization and distributed control problems are studied in this section. For the first time, the necessary and sufficient stabilization conditions are obtained in Theorem \ref{thsta}, which have not been derived before \cite{mm1,mm2}.
\end{remark}

\begin{remark}
 In Theorem \ref{thsta}, on one hand, the definition of the Lyapunov function \eqref{lf} is skillful, which is defined with the optimal cost function, see \eqref{ocf}. On the other hand, we have shown that the optimal distributed controller \eqref{stc} also stabilizes the subsystem \eqref{sys1}, i.e., both the stabilization problem and distributed controller design problem for the infinite horizon have been solved.
\end{remark}

\section{Numerical Examples}
In this section, the numerical example will be provided to illustrate the correctness of the main results in this paper.

We consider 5 subsystems (i.e., $v=5$), and the stabilizaiton property will be studied as below.

For system \eqref{sys1}, \eqref{bar} and cost function \eqref{cost2}, the following coefficients are considered:
\begin{align*}
  A&=2,~B=1,~\bar{F}=-1.5,~Q=R=1,v=5,\\
  \mu_1&=0.3,~\mu_2=0.2,~\mu_3=0.3,~\mu_4=0.1,~\mu_5=0.4.
\end{align*}

By solving the AREs \eqref{are1}-\eqref{are2}, $P$ and $\bar{P}$ are presented as:
$$P=4.2361>0,\quad \bar{P}=0.0972>0. $$
Since the condition $P+\bar{P}\geq P>0$ is satisfied, from Theorem \ref{thsta} we know the system \eqref{sys1} can be stabilized with the distributed controller \eqref{stc}. Then the gain matrices $K,\bar{K}$ in \eqref{kkk} can be calculated as:
$$K=-1.6180,\quad \bar{K}=0.1180.$$
Thus, the distributed control $u_k^i,i=1,\cdots,5$ can be given as:
\begin{align}\label{discon}
  u_k^1&=-1.6180x_k^1+0.0908\bar{x}_k, u_k^2=-1.6180x_k^2+0.0605\bar{x}_k,\notag\\
  u_k^3&=-1.6180x_k^3+0.0908\bar{x}_k, u_k^4=-1.6180x_k^4+0.0303\bar{x}_k,\notag\\
  u_k^5&=-1.6180x_k^5+0.1210\bar{x}_k.
\end{align}

Moreover, we denote the initial state for the subsystems as:
$$x_0^1=3, x_0^2=2,x_0^3=1,x_0^4=4,x_0^5=5,\Rightarrow~\bar{x}_0=4.$$

With the stabilizing controller \eqref{discon}, the regulated system state for the subsystems is depicted in Figure \ref{fig1}.
\begin{figure}[htbp]
  \centering
  \includegraphics[width=0.35\textwidth]{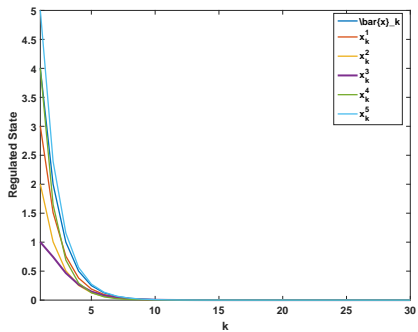}
  \caption{The stabilization of the subsystems with the designed distributed controller \eqref{discon}.}\label{fig1}
\end{figure}

As expected, the state of subsystems converges to 0 with the designed stabilizing controller \eqref{discon}.

\section{Conclusions}
In this paper, we have studied the discrete-time distributed control and stabilization problems for large scale systems with constraints. By developing the maximum principle and decoupling the associated system of the FBDE, we have obtained the optimal distributed controller, and the computation of the distributed controller is independent of the numbers of the subsystems. For infinite horizon case, we have derived the necessary and sufficient stabilization conditions for the first time. For future research, we will extend our work to stochastic systems and systems with delay constraints.

\ifCLASSOPTIONcaptionsoff
  \newpage
\fi

\end{document}